\documentclass[12pt,twoside]{amsart}
\usepackage{latexsym,amsfonts,amsmath,amssymb}
\usepackage{enumerate}
\usepackage{color}

\usepackage{tikz}
\usetikzlibrary{matrix}

\newtheorem{Th}{Theorem}[section]

\newtheorem{Ex}[Th]{Example}
\newtheorem{Lemma}[Th]{Lemma}
\newtheorem{Def}[Th]{Definition}
\newtheorem{Prop}[Th]{Proposition}
\newtheorem{Cor}[Th]{Corollary}

\newcommand{\Li}{\mathcal{L}}

\def\R{\mathbb R}
\def\C{\mathbb C}
\def\N{\mathbb N}

\def\Z{\mathbb Z}

\newcommand{\D}{\mathbb{D}}

\newcommand{\Sch}{\mathcal{S}}

\newcommand{\beqsn}{\arraycolsep1.5pt\begin{eqnarray*}}
\newcommand{\eeqsn}{\end{eqnarray*}\arraycolsep5pt}
\newcommand{\beqs}{\arraycolsep1.5pt\begin{eqnarray}}
\newcommand{\eeqs}{\end{eqnarray}\arraycolsep5pt}

\title{A note on supercyclic operators in locally convex spaces}

\author{Angela A. Albanese}
\address{
Dipartimento di Matematica e Fisica ``E. De Giorgi''\\
Universit\`a del Salento\\
Via per Arnesano, C.P. 193\\
73100 Lecce\\
Italy}
\email{angela.albanese@unisalento.it}

\author{David Jornet}
\address{
Instituto Universitario de Matem\'atica Pura y Aplicada IUMPA\\
Universitat Po\-li\-t\`ecni\-ca de Val\`encia\\
Camino de Vera, s/n\\
E-46071 Valencia\\
Spain}
\email{djornet@mat.upv.es}

\begin{document}


\begin{abstract}
We treat some questions related to supercyclicity of continuous  linear operators when acting in locally convex spaces. We extend results of Ansari and Bourdon and consider doubly power bounded operators in this general setting.  Some examples are given.
\end{abstract}

\maketitle

\section{Introduction and preliminaries}

Let $X$ be a {\em separable} locally convex Hausdorff space (lcHs) and let $\Gamma_{X}$ be the family of all continuous seminorms on $X$. We denote by $\mathcal{L}(X)$ the space of all linear and continuous operators $T:X\to X$. We say, as in the Banach case,  that an operator $T\in\Li(X)$, where $X$ is a lcHs, is  \emph{supercyclic} if there exists $x\in X$ such that the set $\{\lambda T^{n}x:\, \lambda\in \C,\, n\in\N_0\}$ is dense in $X$. In this case, any such a vector is called a \emph{supercyclic vector} for $T$. 

In a recent paper, Aleman and Suciu~\cite{AS} study ergodic theorems for a large class of operator means when the operators act in a Banach space. In particular, they extend a previous result of Ansari and Bourdon~\cite{AB} about power bounded and supercyclic operators on Banach spaces. Motivated by these previous works, we study supercyclic operators acting in a locally convex space and extend some of the results in \cite{AB} to this general setting. Indeed, in Section 2, we extend Theorems 2.1 and 2.2 of \cite{AB}. We have to mention that there is a very general version of \cite[Theorem 3.2]{AB} in \cite{BFS}, and a version of this result for locally convex spaces can be found in \cite[Proposition 1.26]{BM} (see Theorem~\ref{Ansari-th-3.2}).

It is mentioned in \cite{AB}  that ``No isometry on the Banach space $X$ can be supercyclic''. In Section 3, we present some results in this direction when the operators act in the more general setting of a lcHs. Let $X$ be a lcHs and let $\Gamma_{X}$ be the family of all continuous seminorms on $X$.  We say that a subfamily $\Gamma\subseteq \Gamma_{X}$ \emph{defines} or \emph{generates} the topology of $X$ if for every $q\in\Gamma_X$ there exist $p\in\Gamma$ and $\lambda>0$ such that $q\le \lambda p$ (i.e., $q(x)\le \lambda p(x)$ for all $x\in X$).  An operator $T\in \mathcal{L}(X)$ is said to be a $\Gamma$-isometry for some $\Gamma\subseteq \Gamma_{X}$ generating the lc-topology of $X$ if $p(Tx)=p(x)$ for all $p\in\Gamma$ and $x\in X.$ We prove that if $T\in\Li(X)$ is \emph{bijective}, then $T$ is a  $\Gamma$-isometry (for some $\Gamma\subseteq \Gamma_{X}$ generating the lc-topology of $X$) if and only if $T$ is \emph{doubly power bounded} (see Definition~\ref{doubly-pbd}).  There is a large literature of doubly power bounded operators in Banach spaces; we refer to \cite{AA,Lo}, for instance. 

Finally, in Section 4 we use the results of previous sections to give examples of operators that are non-supercyclic, or even of operators which are power bounded and supercyclic or power bounded and non-supercyclic,  acting in Banach and in (non-normable) Fr\'echet spaces. The examples should be compared with \cite{Mu1,Mu2}.

\section{Supercyclic operators in locally convex spaces}

The aim of this section is to extend to the setting of lcHs'  some results about supercyclic operators  due to Ansari and Bourdon \cite{AB}. Let $X$ be a lcHs. We say that an operator $T\in\Li(X)$ is \emph{power bounded} if the sequence $(T^{n})_{n}$ of powers of $T$ is equicontinuous, i.e., for all $p\in\Gamma_X$ there exists $q\in \Gamma_X$ such that  $p(T^{n}x)\le  q(x)$ for all $n\in\N$ and $x\in X$.
\begin{Lemma}
\label{lemma-gamma-isometry}
Let $X$ be a lcHs with $\dim X\ge 2$ and let $T\in\Li(X)$. If $T$ is a $\Gamma$-isometry for some $\Gamma\subseteq \Gamma_{X}$ generating the lc-topology of $X$, then $T$ cannot be a supercyclic operator.
\end{Lemma}

\begin{proof}
Suppose that there exists $y\neq 0$ such that $\{\lambda T^{n}y:\, \lambda\in \C,\, n\in\N_0\}$ is dense in $X$. 
Observe that the vectors $y$ and $Ty$ are linearly independent since, if this is not the case, as $T$ is supercyclic, $\{\lambda y\,:\,\lambda\in\C\}$ is dense and closed in $X$, but this is not possible because $\dim X\ge 2$. Hence, by Hahn-Banach theorem we can find $u,v\in X'$ such that $u(y)=1, u(Ty)=0$ and $v(y)=0, v(Ty)=1$. We denote $q=\max\{|u|,|v|\}\in \Gamma_{X}$. Since $\Gamma$ is generating the locally convex topology of $X$, there exist $p\in \Gamma$ and $\lambda>0$ such that $q\le \lambda p$. Now, we consider the quotient space $\big(\frac{X}{{\rm Ker}\, p},\hat{p}\big)$ and  denote by $Q_p:X\to \frac{X}{{\rm Ker}\, p}$ the canonical quotient map, and by $\hat{p}:\frac{X}{\ker p}\to [0,+\infty[$  the norm $\hat{p}(Q_{p}x):=p(x)$, which is well-defined because if $z\in \ker p$ then for every $x\in X$, $p(x+z)=p(x)$.  Then, $\dim \frac{X}{\ker  p}\ge 2$. In fact, if there is $\mu\in\C$ such that $y+\ker p=\mu \cdot Ty+\ker p$, then $y=\mu\cdot Ty+z$ for some $z\in \ker p$, which implies:
$$
1=u(y)=\mu \cdot u(Ty)+u(z)=0,
$$
a contradiction.

Now, we prove that  there is an isometry $T_p:\frac{X}{\ker p}\to \frac{X}{\ker p}$ satisfying $T_pQ_p=Q_pT$. Indeed, $T_p$ is well-defined because $Q_p(x-y)=0$ implies that $x-y\in{\ker p}$ and hence, $p(T(x-y))=p(x-y)=0$. Accordingly,  $T(x-y)\in {\ker p}$ and so $Q_pT(x-y)=0$.   On the other hand, for each $x\in X$, we have, by the definition of $\hat{p},$ $\hat{p}(T_pQ_px)=\hat{p}(Q_px)$. 
This means that $T_p$ is an isometry from $\frac{X}{\ker p}$ into itself.  It follows that $T_p$ extends to an isometry $\widetilde{T}_p$ on $\big(\frac{X}{\ker p},\hat{p}\big)^{\sim} =:\widetilde{X}_p$ into itself, where $\widetilde{X}_p$ is the Banach completion  of ${X}_p$. 

Next, we observe that $Q_py$ is also a supercyclic vector for $\widetilde{T}_p$. In fact, for each $n\in\N$, we have
\begin{eqnarray*}
T^n_p Q_py &=& T^{n-1}_p T_p Q_py= T^{n-1}_p Q_pTy= T_p^{n-2} (T_pQ_p) Ty\\
&=& T^{n-2}_p Q_p T^2 y=\ldots =Q_p T^n y.
\end{eqnarray*}
Hence, 
$$
Q_p\big(\{\lambda T^{n}y:\, \lambda\in \C,\, n\in \N_0\}\big)=\big\{\lambda T^n_p(Q_py):\ \lambda\in\C,\ n\in \N_0\big\}.
$$
Since $Q_p:X\to \widetilde{X}_p$ is continuous with dense range, it follows that $\big\{\lambda T^n_p(Q_py):\ \lambda\in\C,\ n\in\N_0\big\}$ is also dense in $\widetilde{X}_p$. This shows that $Q_py$ is a supercyclic vector for $\widetilde{T}_p$ on the Banach space $\widetilde{X_p}$. This is a contradiction because $\widetilde{T}_p$ is an isometry; see \cite[Theorem 2.1]{AB}.
\end{proof}

\begin{Th}
\label{ansari-bourdon1}
Let $X$ be a lcHs and $T\in\Li(X)$. Suppose the following properties are satisfied.
\begin{itemize}
\item[\rm (i)] The operator $T$ is power bounded, and
\item[\rm (ii)] For each $x\in X\setminus\{0\}$, $T^{n}x\nrightarrow 0$ in $X$ as $n\to \infty$.
\end{itemize}
Then $T$ has no supercyclic vectors.
\end{Th}

\begin{proof} As in the proof of \cite[Theorem 2.1]{AB}, we fix a linear functional
 $F:\ell^{\infty}\to \R$   with the following properties:
\begin{itemize}
\item[(1)] For every $(x_n)_n,\, (y_n)_n\in\ell^\infty$, if $x_n\leq y_n$ for all $n\in\N$, then $F((x_n)_n)\leq F((y_n)_n)$,
\item[(2)] For  every $(x_n)_n\in\ell^\infty$, $F((x_n)_n)=F((x_{n+1})_n)$,
\item[(3)] $F((x_n)_n)$ is the limit of a subsequence of $(\frac{x_1+\cdots+x_n}{n})_n$.
\end{itemize} 

For each  $p\in\Gamma_X$ we define
$$
\gamma_{p}(x):=F\big((p(T^{n}x))_{n}\big),\quad x\in X.
$$
Then $\gamma_p$ is well-defined by assumption (i). Actually,  $\gamma_{p}$ is a seminorm on $X$ as it  easily follows from the linearity of $F$ combined with its property (1) and with the fact that $p$ is a seminorm. But, $\gamma_p$ is not a norm in general.  So, $\big(X,(\gamma_{p})_{p\in\Gamma}\big)$ is a locally convex space. Moreover, $\big(X,(\gamma_{p})_{p\in\Gamma}\big)$ is Hausdorff because if $x\ne 0$, then assumption (ii) ensures that $T^{n}x\nrightarrow 0$ in $X$ as $n\to \infty$ and hence, $p(T^{n}x)\nrightarrow 0$ as $n\to\infty$ for some $p\in\Gamma_X$. So, there are $(n_{j})_{j}\subset \N$ which tends to infinity and $\delta>0$ such that 
$$
p(T^{n_{j}}y)>\delta>0,\ \ j\in\N.
$$
Since $T$ is power bounded, given this seminorm $p$ there exists  $q\in \Gamma_{X}$ such that
$$
p(T^{n+m}x)\le  q(T^{m}x),\ x\in X,\ n,m\in\N.
$$
Then, fixed $n\in\N$ we find $j\in\N$ with  $ n<n_{j}$. So,
$$
\delta<p(T^{n_{j}}y)=p(T^{n+(n_{j}-n)}y)\le q(T^{n}y).
$$
Therefore, $q(T^{n}y)>\delta>0$ for all $n\in\N.$ Now, property (3) of $F$ shows that $\gamma_q(x)>0$.

We now observe that for every $p\in\Gamma_{X}$ and  $x\in X$, the property (2) of $F$ implies that 
\begin{eqnarray*}
\gamma_p(x)&=& F\big((p(T^{n}x))_{n}\big)=F\big((p(T^{n+1}x))_{n}\big)\\
&=& F\big((p(T^{n}(Tx)))_{n}\big)=\gamma_{p}(Tx).
\end{eqnarray*}
It follows  that $T$ is a $\Gamma$-isometry from $\big(X,(\gamma_{p})_{p\in\Gamma_X}\big)$ into itself. This fact implies that $T$ cannot be a supercyclic operator from $X$ into itself. To see this, we first
note that the inclusion 
$$
i:X	\to \big(X,(\gamma_{p})_{p\in\Gamma_X}\big)
$$
is continuous. Indeed, fixed $p\in\Gamma_X$, by assumption $(i)$ there exist $q\in\Gamma_X$ such that  $p(T^{n}x)\le  q(x)$ for all $x\in X$ and $n\in \N$. It follows for each $x\in X$ that
$$
\gamma_p(x)=F\big((p(T^{n}x))_{n}\big)\le F\big((q(x))_{n}\big)= F(\mathbf{1})q(x).
$$
%
The continuity of $i$  imply that if $x\in X$ is a supercyclic vector for $T$ in $X$ then $x$ is also a  supercyclic vector for $T$ in $\big(X,(\gamma_{p})_{p\in\Gamma_X}\big)$; this is a contradiction by Lemma \ref{lemma-gamma-isometry} because $T$ is a $\Gamma$-isometry from $\big(X,(\gamma_{p})_{p\in\Gamma_X}\big)$ into itself.
\end{proof}

%
We observe that the next result improves Theorem~\ref{ansari-bourdon1}.
\begin{Th}
\label{ansari-bourdon2}
Let $X$ be a lcHs and let $T\in\Li(X)$. If $T$ is power bounded and supercyclic, then $T^{n}x\to 0$ in $X$ as $n\to\infty$ for all $x\in X$.
\end{Th}

\begin{proof}
We first prove the following claim: if $y\in X$ is a supercyclic vector for $T$, then $T^{n}y\to 0$ in $X$ as $n\to \infty.$
We argue by contradiction and assume that  there is $y\in X$, $y\ne 0$, so that $\{\lambda T^{n}y:\ \lambda\in \C, n\in\N_0\}$ is dense in $X$ but, $T^{n}y\nrightarrow 0$ in $X$ as $n\to\infty.$

Since $T^{n}y\nrightarrow 0$ as $n\to\infty,$ proceeding as in the proof of Theorem~\ref{ansari-bourdon1}, we show that there are some seminorm $q\in\Gamma$ and $\delta>0$ such that  $q(T^{n}y)>\delta>0$ for all $n\in\N.$

Since $T$ is power bounded and supercyclic, from Theorem~\ref{ansari-bourdon1} it follows that there is $v\ne 0$  such that $T^{n}v\to 0$ in $X$. Since $v\ne 0$, there exists $r\in\Gamma_X$ for which $r(v)\ne 0$ because $X$ is Hausdorff. On the other hand, there is $s\in\Gamma_X$ so that $\max\{q,r\}\le s$. Hence, $s(v)\ne 0$ and $s(T^{n}y)>\delta$ for all $n\in\N.$

For simplicity, we denote the seminorm $s$ again by $q$. Now, let $r\in\Gamma_X$, and $r\ge q$ so that $q(T^{n}x)\le r(x)$ for all $n\in\N$ and $x\in X$. We may assume without loss of generality that $q(v)=1$. Since $y$ is a supercyclic vector for $T$, there exist $(c_{j})_{j}\subset \C$ and $(n_{j})_{j}\subset \N$ such that
$$
r(c_{j}T^{n_{j}}y-v)\to 0\ \mbox{as}\ j\to\infty.
$$
It follows that there is $k\in\N$ such that for all $j\ge k$ we have
$$
q(c_{j}T^{n_{j}}y-v)\le r(c_{j}T^{n_{j}}y-v)<\frac{1}{2}.
$$
Therefore, for all $j\ge k$ we have
$$
q(c_{j}T^{n_{j}}y)=q[v-(v-c_{j}T^{n_{j}}y)]\geq q(v)- q(c_{j}T^{n_{j}}y-v)>1-\frac{1}{2}=\frac{1}{2}.
$$
So, for all  $j\ge k$,
$$
\frac{1}{2}<q(c_jT^{n_{j}}y)=|c_{j}| q(T^{n_{j}}y)\le |c_{j}| r(y),
$$ 
which implies that $|c_{j}|>\frac{1}{2r(y)}$ for all  $j\ge k$.

Let $\varepsilon=\frac{\delta}{3r(y)}.$ Since $r(c_{j}T^{n_{j}}y-v)\to 0$ as $j\to\infty$, we can find $h\ge k$ such that 
$$
r(c_{h}T^{n_{h}}y-v)<\frac{\varepsilon}{2}.
$$
But $T^{n}v\to 0$ in $X$ as $n\to\infty$ and so, we can find $m\in\N$ with
$$
q(T^{m}v)\le r(T^{m}v)<\frac{\varepsilon}{2}.
$$
Now, we observe that,
\begin{eqnarray*}
q(c_{h} T^{n_{h}+m}y-T^{m}v)&=&q(T^{m}(c_{h}T^{n_{h}}y-v))\\
&\le & r(c_{h}T^{n_{h}}y-v)<\frac{\varepsilon}{2},
\end{eqnarray*}
and that
$$
q(c_{h}T^{n_{h}+m}y)=|c_{h}|q(T^{n_{h}+m}y)>\delta \frac{1}{2r(y)}.
$$
Consequently,
\begin{eqnarray*}
\frac{\delta}{2r(y)}&<&q(c_{h} T^{n_{h}+m}y)\le q(c_hT^{n_{h}+m}y-T^{m}v)+q(T^{m}v)\\
&< & \frac{\varepsilon}{2}+\frac{\varepsilon}{2}=\varepsilon=\frac{\delta }{3r(y)};
\end{eqnarray*}
a contradiction.

We have proved that $T^ny\to 0$ in $X$ as $n\to\infty$ whenever $y\in X$ is a supercyclic vector for $T$. But, the set of all supercyclic vectors for $T$ is dense in $X$. Indeed, if $y\in X$ is a supercyclic vector for $T$, then also $cT^ky$ is a supercyclic vector for $T$ for all $c\in \C\setminus \{0\}$ and $k\in\N$, as it is easy to see. Now, the density in $X$ of  the set of all supercyclic vectors for $T$ and the equicontinuity of $(T^n)_n$ imply that $T^nx\to 0$ in $X$ as $n\to\infty$ for all $x\in X$. In particular, we get a contradiction with Theorem~\ref{ansari-bourdon1}.
\end{proof}

We finish this section with an extension of \cite[Theorem 3.2]{AB}. An even more general version of this result can be found in \cite[Theorem 2.1]{BFS}. We recall that given $T\in\Li(X)$, the \emph{point spectrum} $\sigma_{p}(T)$ of $T$ consists of all $\lambda\in \C$ such that the operator $\lambda I-T$ is not injective, where $I:X\to X$ denotes the identity operator. For the proof, we refer to \cite[Proposition 1.26]{BM}. 

\begin{Th}\label{Ansari-th-3.2} Let $X$ be a lcHs and $T\in \Li(X)$. If 
$T$ is a supercyclic operator, then the point spectrum of the adjoint operator $T'$ of $T$, $\sigma_p(T')$, contains at most one point.
\end{Th}

\section{Doubly power bounded operators}

In this section, we characterize the operators $T\in\Li(X)$ which are bijective on a locally convex space $X$ such that there is $\Gamma\subseteq \Gamma_X$ defining the topology of $X$ such that $T$ is a $\Gamma$-isometry. The following definition extends the analogous one for Banach spaces (see, for instance, \cite{AA,Lo}).

\begin{Def}\label{doubly-pbd}
An operator $T\in\Li(X)$ is \emph{doubly power bounded} if it is  bijective and  $(T^{k})_{k\in\Z}$ is equicontinuous in $\Li(X)$.
\end{Def}
Observe that if a bijective operator $T\in\Li(X)$ is doubly power bounded then, in particular, $T^{-1}\in\Li(X).$ However, in a locally convex space the open mapping theorem does not hold in general: there is a locally convex space $X$ and a continuous, linear and bijective map $T\in\Li(X)$ which is not open. For instance, consider in $c_{00}$ (the space of eventually null sequences) the norm induced by $c_{0}$ (the sup norm) and the diagonal operator $Te_{i}=i^{-1}e_i,$ $i=1,2,\ldots,$ where $(e_{i})_{i}$ is the canonical basis. The operator $T$ is bijective and continuous on $c_{00}$ but $T^{-1}$ is not continuous since the sequence $(i^{-1/2}e_{i})_{i}$ tends to zero in $c_{00}$ but $(T^{-1}(i^{-1/2}e_{i}))_{i}=(i^{1/2}e_{i})_{i}$, which is not bounded.

\begin{Prop}\label{dpbd-gamma-iso}
An operator  $T\in\Li(X)$ is doubly power bounded if and only if it is bijective and there is $\Gamma\subseteq \Gamma_X$ defining the topology of $X$ such that $T$ is a $\Gamma$-isometry.
\end{Prop}

\begin{proof}
Assume first that $T$ is doubly power bounded. Given $q\in \Gamma_X$, define
$$
r_{q}(x):=\sup_{k\in\Z} q(T^{k}x).
$$
Clearly, taking $k=0$, we have
\begin{equation}\label{rq1}
q(x)\le r_{q}(x), \mbox{ for all }x\in X.
\end{equation}
  On the other hand, since $(T^{k})_{k\in\Z}$ is equicontinuous, given $q\in \Gamma_X$ there is $p\in \Gamma_X$ such that $q(T^{k}x)\le p(x),$ for all $x\in X$ and $k\in\Z.$ This implies that 
\begin{equation}\label{rq2}
r_{q}(x)\le p(x),\quad x\in X.
\end{equation} 
In particular, $r_{q}(x)<\infty$ for all $x\in X$. Moreover, $r_{q}\in \Gamma_X$ as it is easily seen from the facts that $T^{k}$ is linear for all $k\in\Z$ and \eqref{rq2}. We consider 
$$
\Gamma:=\{r_{q}\,:\,q\in \Gamma_X\}.
$$
By \eqref{rq1} \and \eqref{rq2}, $\Gamma$ defines the topology of $X$. We observe that $T$ is a $\Gamma$-isometry since $$r_{q}(Tx)=\sup_{k\in\Z} q(T^{k} Tx)=\sup_{k\in\Z} q(T^{k} x)=r_{q}(x).$$

Now, suppose that $T\in \Li(X)$ is a bijective $\Gamma$-isometry for a set $\Gamma\subseteq \Gamma_X$ defining the topology of $X$. By assumption there exists $T^{-1}:X\to X$ linear. Since $p(Tx)=p(x)$ for all $x\in X$ and $p\in \Gamma$, we have $p(T^{-1}x)=p(x)$ for all $x\in X$ and $p\in \Gamma$. Since $\Gamma$ defines the topology of $X$, $T^{-1}$ is continuous, and moreover,
$$
p(T^{k}x)=p(x),\quad x\in X,\ p\in\Gamma.
$$
Now, we take $q\in \Gamma_X$ arbitrary. There is $p\in\Gamma$, $\lambda>0$ such that $q\le \lambda p.$ For $k\in\Z$ and $x\in X$ we get
$$
q(T^{k}x)\le \lambda p(T^{k}x)=\lambda p(x).
$$
This implies that $(T^{k})_{k\in\Z}$ is equicontinuous.
\end{proof}

\begin{Cor}\label{ext-doubly-pbd}
If $\dim X\ge 2$ and $T\in \Li(X)$ is doubly power bounded, then $T$ is not supercyclic.
\end{Cor}
\begin{proof}
This follows from Proposition~\ref{dpbd-gamma-iso} and Lemma~\ref{lemma-gamma-isometry}.
\end{proof}

\section{Examples}

We present different examples of power bounded and supercyclic or non-supercyclic operators in a Banach space or in non-normable Fr\'echet spaces. First of all, we observe that every $\Gamma$-isometry, for some $\Gamma$ generating the lc-topology of $X$, is obviously a power bounded operator. The first example is well known.

\begin{Ex}{\rm Our first example is \cite[Example 1.15]{BM}, which is a positive example in Banach spaces. Let $B_{\mathbf{w}}$ be the weighted backward shift in $\ell^{2}(\N)$. This operator is defined by $B_{\mathbf{w}}(e_{1})=0$ and $B_{\mathbf{w}}(e_{n})=w_{n}e_{n-1}$ for $n\ge 2$ where $(e_{n})_{n\in\N}$ is the canonical basis in $\ell^{2}(\N)$ and $\mathbf{w}=(w_{n})_{n\ge 2}$ is a bounded sequence of positive numbers. By \cite[Theorem 1.14]{BM}, $B_{\mathbf{w}}$ is supercyclic. Moreover, if the sequence $\mathbf{w}$ satisfies $w_{n}\le 1$ for all $n\ge 2$, it is easy to see that $B_{\mathbf{w}}$ is also power bounded.

}
\end{Ex}

\begin{Ex}{\rm
Given an open and connected (=\emph{domain}) subset $U$ in $\C^{d}$ we denote 
$$
H(U)=\{f:U\to\C,\ f\mbox{ holomorphic in }U\}.
$$
A composition operator $C_{\varphi}:H(U)\to H(U)$ with (holomorphic) {\em symbol}  $\varphi:U\to U$ 
is the linear and continuous operator given  by $C_{\varphi}(f)(z):=f(\varphi(z))$ for $z\in U$ and $f\in H(U)$.

\begin{itemize}
\item[a)] Let $U=\D$ be the open unit disk in $\C$ and $\Gamma$  the family of seminorms $\{p_{k}\,:\,k\in\N\}$ where $p_{k}(f):=\sup_{|z|\le 1-\frac{1}{k}} |f(z)|$, for $k\in\N$ and $f\in H(\D)$.  If $\theta\in\C$ with $|\theta|=1$, the composition operator $C_{\varphi}:H(\D)\to H(\D)$ with symbol $\varphi(z):=\theta z$ (a \emph{rotation}) clearly satisfies 
$$
p_{k}(C_{\varphi}f)=p_{k}(f),\quad f\in H(\D), \quad k\in\N.
$$
Hence, $C_{\varphi}$ is a $\Gamma$-isometry. Moreover, it is bijective and doubly power bounded.  Since $\Gamma$ generates the lc-topology of $H(\D),$ the composition operator $C_{\varphi}$ with symbol given by a rotation cannot be supercyclic in the (non-normable) Fr\'echet space $H(\D)$. 

\item[b)] On the other hand, Bonet and Doma\'nski~\cite{BD} characterized,  in terms of its symbol,  when the composition operator $C_{\varphi}:H(U)\to H(U)$ is power bounded in a very general situation (namely, when $U$ is a Stein manifold), proving that the composition operator is power bounded if and only if it is mean ergodic, i.e., the sequence of Ces\`aro means 
$(\frac{1}{n}\sum_{j=0}^{n-1} C_{\varphi}^{n}(f))_{n}$
converges in $H(U)$ for each $f\in H(U)$.   Using their results, we can give an example in a very general setting:  let $U$ be  a topologically contractible bounded strongly pseudoconvex domain in $\C^{d}$ with $\mathcal{C}^{3}$ boundary and $\varphi:U\to U$ a holomorphic symbol with a fixed point (for example, when $d=1$ and $U=\D$, the open unit disk).  Then by \cite[Corollary 1]{BD} the composition operator $C_{\varphi}:H(U)\to H(U)$ is power bounded and, hence, it cannot be supercyclic. In fact, if $C_{\varphi}$ is supercyclic, by Theorem~\ref{ansari-bourdon2},  $C_{\varphi}^{n}(f)=f\circ\varphi^{n}\rightarrow 0$ in $H(U)$ for each $f\in H(U)$, but this is not true for $f\equiv 1$. We observe that there are holomorphic symbols $\varphi$ such that  $C_{\varphi}$ has dense range. For instance, when $\varphi$ is an automorphism. We can find similar examples in spaces of real analytic functions; see, e.g., \cite[Corollary 2.5]{BD-ra}.
\end{itemize}

}
\end{Ex}

The following simple example is related to Fr\'echet sequence spaces.

\begin{Ex}{\rm 
We consider a K\"othe sequence space $\lambda_{p}(A)$ with associated matrix $A=(a_{n}(i))_{n,i\in\N}$, with $1\le p\le \infty$. For the precise definition see, for instance, at the beginning of chapter 27 of \cite{MV}; there, the notation is $a_{n}(i)=a_{i,n}$ for the elements of the K\"othe matrix. Given a sequence $(b_{n})_{n}\subseteq \C$ and $\Gamma$ the fundamental sequence of seminorms defined in \cite{MV}, it is easy to see that the diagonal operator
$$
T_{b}:\lambda_{p}(A)\to\lambda_{p}(A),\quad T_{b}(x)=(b_{n}x_{n})_{n},
$$
is a $\Gamma$-isometry if and only if $|b_{n}|=1$ for all $n\in\N$. Moreover, it is doubly power bounded also. Hence, in this case, by Lemma~\ref{lemma-gamma-isometry}, $T_{b}$ cannot be supercyclic. 
}
\end{Ex}

Now, we find an operator that is power bounded and not supercyclic on a Fr\'echet space; see \cite{AS,Mu1,Mu2} for different situations in Banach spaces. This example shows that for a power bounded operator, the thesis in Theorem~\ref{ansari-bourdon2} is not sufficient for the operator  to be supercyclic.

\begin{Ex}\label{ejemplo-Pepe}{\rm 
It is known from \cite[Proposition 4.3]{BBF} that the integration operator 
$$
Jf(z):=\int_{0}^{z} f(\zeta)d\zeta
$$
is power bounded in $H(\C)$ or in $H(\D)$ and, moreover, $J^{n}f$ tends to $0$ as $n$ tends to infinity in the compact-open topology for every $f$ in these spaces. However, the integration operator $J$ is not supercyclic in $H(\C)$ or in $H(\D)$, since it does not have dense range in these spaces.  
}
\end{Ex}
Our last example also shows that the thesis in Theorem~\ref{ansari-bourdon2} is necessary but not sufficient for a power bounded operator to be supercyclic in the Schwartz class $\Sch(\R)$ of rapidly decreasing functions in one variable. We give  examples of  power bounded and non supercyclic operators which have dense range in $\Sch(\R)$.
\begin{Ex}\label{ejemplo-Schwartz}
{\rm 
 If we consider the Schwartz class $\Sch(\R)$ of rapidly decreasing functions in one variable, the composition operator $C_{\varphi}:\Sch(\R)\to \Sch(\R)$ is well defined and continuous if and only if the symbol $\varphi\in C^{\infty}(\R)$ satisfies some conditions \cite[Theorem 2.3]{GJ-Schwartz}, and $C_{\varphi}$ is never compact. On the other hand, $C_{\varphi}:\Sch(\R)\to \Sch(\R)$ is never supercyclic \cite[Corollary 2.2(1)]{FGJ-Schwartz}, but the authors find examples of symbols (namely, any polynomial of even degree greater than one without fixed points) such that $C_{\varphi}:\Sch(\R)\to \Sch(\R)$ is power bounded, mean ergodic and $(C_{\varphi}^{n})_{n}$ converges pointwise to zero in $\Sch(\R)$ \cite[Theorem 3.11, Corollary 3.12]{FGJ-Schwartz}. The authors also show that if the symbol $\varphi$ is monotonically decreasing and the corresponding composition operator is power bounded then $(C_{\varphi})^{2}=I$, the identity, so in this case $C_{\varphi}$ is surjective, and hence it has also dense range~\cite[Theorem 3.8~(b)]{FGJ-Schwartz}. Moreover, besides $\varphi(x)=-x$ there are many monotonically decreasing symbols $\psi$ such that $(C_{\psi})^{2}=I$~\cite[Example 1]{FGJ-Schwartz}.  }
\end{Ex}

\textbf{Acknowledgements.} We are indebted to Prof.~Jos\'e Bonet for his helpful suggestions on the topic of this paper. The authors were partially supported by the project MTM2016-76647-P. The authors are very grateful to the referee for several valuable comments and remarks, and for the careful reading of the paper.


\begin{thebibliography}{10}

\bibitem{AA}
Y.~A. Abramovich and C.~D. Aliprantis.
\newblock {\em An invitation to operator theory}, volume~50 of {\em Graduate
  Studies in Mathematics}.
\newblock American Mathematical Society, Providence, RI, 2002.

\bibitem{AS}
Alexandru Aleman and Laurian Suciu.
\newblock On ergodic operator means in {B}anach spaces.
\newblock {\em Integral Equations Operator Theory}, 85(2):259--287, 2016.

\bibitem{AB}
Shamim~I. Ansari and Paul~S. Bourdon.
\newblock Some properties of cyclic operators.
\newblock {\em Acta Sci. Math. (Szeged)}, 63(1-2):195--207, 1997.

\bibitem{BM}
Fr\'ed\'eric Bayart and \'Etienne Matheron.
\newblock {\em Dynamics of linear operators}, volume 179 of {\em Cambridge
  Tracts in Mathematics}.
\newblock Cambridge University Press, Cambridge, 2009.

\bibitem{BBF}
Mar\'\i a~Jos\'e Beltr\'an, Jos\'e Bonet, and Carmen Fern\'andez.
\newblock Classical operators on the {H}\"ormander algebras.
\newblock {\em Discrete Contin. Dyn. Syst.}, 35(2):637--652, 2015.

\bibitem{BD}
Jos\'e Bonet and Pawe\l\ Doma\'nski.
\newblock A note on mean ergodic composition operators on spaces of holomorphic
  functions.
\newblock {\em Rev. R. Acad. Cienc. Exactas F\'\i s. Nat. Ser. A Math. RACSAM},
  105(2):389--396, 2011.

\bibitem{BD-ra}
Jos\'e Bonet and Pawe\l\ Doma\'nski.
\newblock Power bounded composition operators on spaces of analytic functions.
\newblock {\em Collect. Math.}, 62(1):69--83, 2011.

\bibitem{BFS}
P.~S. Bourdon, N.~S. Feldman, and J.~H. Shapiro.
\newblock Some properties of {$N$}-supercyclic operators.
\newblock {\em Studia Math.}, 165(2):135--157, 2004.

\bibitem{FGJ-Schwartz}
Carmen Fern\'andez, Antonio Galbis, and Enrique Jord\'a.
\newblock Dynamics and spectra of composition operators on the {S}chwartz
  space.
\newblock {\em J. Funct. Anal.}, 274(12):3503--3530, 2018.

\bibitem{GJ-Schwartz}
Antonio Galbis and Enrique Jord\'a.
\newblock Composition operators on the {S}chwartz space.
\newblock {\em Rev. Mat. Iberoam.}, 34(1):397--412, 2018.

\bibitem{Lo}
Edgar~R. Lorch.
\newblock The integral representation of weakly almost-periodic transformations
  in reflexive vector spaces.
\newblock {\em Trans. Amer. Math. Soc.}, 49:18--40, 1941.

\bibitem{MV}
Reinhold Meise and Dietmar Vogt.
\newblock {\em Introduction to functional analysis}, volume~2 of {\em Oxford
  Graduate Texts in Mathematics}.
\newblock The Clarendon Press, Oxford University Press, New York, 1997.
\newblock Translated from the German by M. S. Ramanujan and revised by the
  authors.

\bibitem{Mu1}
V.~M\"uller.
\newblock Power bounded operators and supercyclic vectors.
\newblock {\em Proc. Amer. Math. Soc.}, 131(12):3807--3812, 2003.

\bibitem{Mu2}
V.~M\"uller.
\newblock Power bounded operators and supercyclic vectors. {II}.
\newblock {\em Proc. Amer. Math. Soc.}, 133(10):2997--3004, 2005.

\end{thebibliography}

\end{document}